\documentclass[11pt]{amsart}
\usepackage[utf8]{inputenc}
\usepackage[english]{babel}
\usepackage[colorlinks=true]{hyperref}
\usepackage{cleveref}  
\usepackage{enumerate}
\usepackage{geometry}

\usepackage{tikz}
\usepackage{tikz-cd}
\usepackage{graphicx}

\usepackage{amsthm, amssymb, amsfonts, amsxtra, amsmath, mathrsfs}
\usepackage{wasysym, mathtools}
\usepackage{stmaryrd}
\usepackage[normalem]{ulem}
\usepackage[all]{xy}

\usepackage{easyReview}  

\usepackage{hyperref}    
\usepackage{cleveref}    

\newtheorem{thm}{Theorem}[section]
\newtheorem*{thm*}{Theorem}
\newtheorem{prop}[thm]{Proposition}
\newtheorem*{prop*}{Proposition}
\newtheorem{corol}[thm]{Corollary}
\newtheorem*{corol*}{Corollary}

\theoremstyle{definition}

\newtheorem{lemma}[thm]{Lemma}
\newtheorem{dfn}[thm]{Definition}
\newtheorem*{dfn*}{Definition}
\newtheorem{rmk}[thm]{Remark}

\DeclareMathOperator{\Fix}{Fix}

\DeclareMathOperator{\hilb}{Hilb}
\DeclareMathOperator{\Bl}{Bl}

\DeclareMathOperator{\Pic}{Pic}
\DeclareMathOperator{\Aut}{Aut}
\DeclareMathOperator{\NS}{NS}
\DeclareMathOperator{\Sing}{Sing}

\DeclareMathOperator{\id}{id}

\DeclareMathOperator{\reg}{reg}
\DeclareMathOperator{\Pn}{\mathbb{P}}
\DeclareMathOperator{\Gr}{Gr}
\DeclareMathOperator{\Sym}{Sym}

\newcommand{\Hilb}[2]{#2^{[#1]}}
\newcommand{\Blow}[2]{\Bl_{#1}#2}

\title{The Fano variety of lines on singular cyclic cubic fourfolds}

\author{Samuel Boissi\`ere}
\address{Samuel Boissi\`ere,
	Laboratoire de Math\'ematiques et Applications,
	UMR 7348 du CNRS,
	B\^atiment H3,
	Boulevard Marie et Pierre Curie,
	Site du Futuroscope,
	TSA 61125,
	86073 Poitiers Cedex 9,
	France}
\email{samuel.boissiere@univ-poitiers.fr}
\urladdr{http://www-math.sp2mi.univ-poitiers.fr/$\sim$sboissie/}

\author{Paola Comparin}
\address{Paola Comparin, 
Departamento de Matem\'atica y Estad\'istica, 
Universidad de la Frontera, 
Av. Francisco Salazar 1145, Temuco, Chile}
\email{paola.comparin@ufrontera.cl}
\urladdr{https://sites.google.com/site/pcompa/}

\author{Lucas Li Bassi}
\address{Lucas Li Bassi, 	Laboratoire de Math\'ematiques et Applications,
	UMR 7348 du CNRS,
	B\^atiment H3,
	Boulevard Marie et Pierre Curie,
	Site du Futuroscope,
	TSA 61125,
	86073 Poitiers Cedex 9,
	France}
\email{lucas.li.bassi@univ-poitiers.fr}
\urladdr{http://https://sites.google.com/view/lucaslibassi/}

\date{\today}

\begin{document}
	\maketitle
	\begin{abstract}
		We study the symplectic resolution of the Fano variety of lines on some singular cyclic cubic fourfolds, i.e. cubic fourfolds arising as cyclic 3:1 cover of $\mathbb{P}^4$ branched along a cubic threefold. In particular we are interested in the geometry of these varieties in the case of cyclic cubic fourfolds branched along a cubic threefold having one isolated singularity of type $A_i$ for $i=2,3,4$. On these symplectic resolutions we find a non-symplectic automorphism of order three induced by the covering automorphism. 
	\end{abstract}
	
	\section{Introduction}
	In the context of complex projective geometry, cubic hypersurfaces have been deeply studied by the mathematical community for many reasons, one of them being their rich geometry. There has been a growing interest over the last fitfty years in one class in particular: cubic fourfolds. One of the reasons why cubic fourfolds are particularly interesting resides in their Hodge structure. Indeed, they are the archetypal example of Fano varieties of $K3$ type (see \cite{fatighenti2022topics} for a survey on the subject). Because of this fact they are deeply related to the world of irreducible holomorphic symplectic (from now on IHS) manifolds. \\
	
	The first example of this relation is via the construction of the Fano variety of lines on a cubic fourfold, which is an IHS manifold of $K3^{[2]}$-type by the well known result of Beauville--Donagie ~\cite{BD}. This construction was crucial in order to establish an isomorphism between the moduli space $\mathcal{C}_3^{sm}$ of smooth cubic threefolds and the moduli space $\mathcal{N}_{\langle 6 \rangle}^{\rho, \xi}$ of fourfolds deformation equivalent to the Hilbert square of a $K3$ surface endowed with a special non-symplectic automorphism of order three as shown by Boissi\`ere--Camere--Sarti~\cite[Theorem 1.1]{BCS}. In \textit{loc. cit.} the authors use the \emph{cyclic} fourfold, i.e. the 3:1 cyclic cover of $\mathbb{P}^4$ branched on a cubic threefold, and the Fano variety of lines on it to show this isomorphism. This provides an interpretation in the framework of IHS manifolds of the results of Allcock--Carlson--Toledo~\cite{ACT}. 
	In the latter, the authors consider also the GIT compactification of $\mathcal{C}_3^{sm}$ providing an extension of the period map also for GIT stable nodal cubic threefolds, i.e. cubic threefolds having only isolated singularities of type $A_k$ for $k=1, \dots, 4$. Also these cases have been given an IHS interpretation in the framework of the construction of moduli spaces, period maps and period domains of nonsymplectic automorphisms on IHS manifolds deformation equivalent to the Hilbert square of a K3 surface, in \cite{BCS} for $i=1$ and \cite{libassi2023git} for $i=2,3,4$.\\
	
	The meaning of the degeneracy of the automorphism is that when the period point goes to the closure of the period domain, the automorphism of the family jumps to another family with a bigger invariant lattice. In view of this, in \cite{boissiere2023fano} the authors studied in detail the geometry of the Fano variety of lines of a cuspidal cyclic cubic fourfold, i.e. a cyclic cubic fourfold having a cubic threefold with one isolated $A_1$ singularity as branch locus.\\
	
	The aim of this paper is to supplement their results studying the Fano variety of lines of a cyclic cubic fourfold having as branch locus a cubic threefold with one isolated $A_i$ singularity with $i=2,3,4$. Therefore, the main theorem of this paper is the following.
	\begin{thm}\label{thm: main SamuelPaola}
		Let $C_i$ be a complex projective cubic threefold having one isolated singularity of type $A_i$ for $i=2, 3, 4$ and let $Y_i$ be its associated cyclic cubic fourfold. Assume that there exist no plane $\Pi\subset Y_i$ such that $\Pi\cap\Sing(Y_i)\neq\emptyset$. Then the Fano variety of lines $F(Y_i)$ of~$Y_i$ admits a unique symplectic resolution by an IHS manifold of $K3^{[2]}$-type $\widetilde{F(Y_i)}$. \\
		Moreover, there exist integral lattices $R_i$ and $T_i$ such that:\begin{enumerate}[i)]
			\item $\Pic\left(\widetilde{F(Y_i)}\right)\simeq R_i$;
			\item there exists a non-symplectic automorphism $\tau\in\Aut\left(\widetilde{F(Y_i)}\right)$ whose invariant sublattice is $H^2(\widetilde{F(Y_i)}, \mathbb{Z})^{\tau^*}\simeq T_i$;
		\end{enumerate}
		with $T_i$ and $R_i$ defined in the following table:
		\begin{center}
			\begin{tabular}{|c|c|c|} 
				\hline
				$i$ & $T_i$ & $R_i$\\
				\hline
				$2$ &  $\langle 6\rangle\oplus A_2$ & $\langle 6\rangle\oplus D_4$\\
				\hline
				$3$ &  $\langle 6\rangle\oplus E_6$ & $\langle 6\rangle\oplus E_6$\\ 
				\hline
				$4$ & $\langle 6\rangle\oplus E_8$ & $\langle 6\rangle\oplus E_8$\\ 
				\hline
			\end{tabular}
		\end{center}
	\end{thm}
	\section*{Acknowledgements}
	We would like to warmly thank Chiara Camere, Enrico Fatighenti, Michele Graffeo, Christian Lehn, Giovanni Mongardi, Alessandra Sarti and Ryo Yamagishi for the helpful insights given at various points of the redaction of this paper. The authors have been partially funded by Programa de cooperación científica ECOS-ANID C19E06 and by the MATH AmSud projet ALGEO 2020-2022.  
	The second author has been partially funded by Universidad de La Frontera, Proyecto DIM23-0001 and the third author was partially funded by the grant ``VINCI 2021-C2-104'' issued by the ``Université Franco Italienne''. 
	\section{Basic facts about symplectic varieties}
	In this section we recall some basic facts about symplectic varieties.\newline
	Let $X$ be a normal complex projective variety and $X^{\reg}$ its regular part. The sheaf $\Omega^{[p]}_X$ of \emph{reflexive holomorphic p-forms} on $X$ is defined as $\iota_*\Omega^{p}_{X^{\reg}}$, where we denote by $\iota$ the inclusion of the regular part $X^{\reg}\subset X$.
	Then a symplectic form on $X$ is a closed reflexive 2-form $\omega$, i.e. a global section of $\Omega^{[2]}_X$, on $X$ which is non-degenerate at each point of $X^{\reg}$.
	\begin{dfn}[\protect{\cite[Definition 1.1]{Beauville_2000}}]
		Assume that a normal projective variety $X$ admits a symplectic form $\omega$. Then $X$ has \emph{symplectic singularities} if for one (hence for every) resolution $f:\widetilde{X}\to X$ of the singularities (i.e. a birational proper map from a smooth variety) of X, the pullback $f^*\omega_{\reg}$ of the holomorphic symplectic form $\omega_{\reg}=\omega|_{X^{\reg}}$ extends to a holomorphic 2-form on $\widetilde{X}$. In this case $X$ is called \emph{symplectic variety}.
	\end{dfn}
	From now on we denote by $X$ a symplectic variety, $\omega$ a symplectic form on it and $\pi:\widetilde{X}\to X$ a resolution of singularities. Then the regular 2-form $\pi^{[*]}\omega$ (see the discussion in \cite{Kebekuspullback} for the definition of pullback of reflexive forms), in general, is degenerate. Therefore we give the following definition.
	\begin{dfn}
		A resolution of singularities $\pi:\widetilde{X}\to X$ is said \emph{symplectic} if $\pi^{[*]}\omega$ is non-degenerate.
	\end{dfn}
	Recall that a birational map $\pi:Y\to X$ between normal irreducible algebraic varieties with canonical bundles $K_X$ and $K_Y$ is called $\emph{crepant}$ if the canonical map $$\pi^*K_X\to K_{Y}$$ defined over the non-singular locus extends to an isomorphism over the whole manifold $Y$. 
	\begin{prop}
		Let $\pi:\widetilde{X}\to X$ be a resolution of singularities.
		The following are equivalent:
		\begin{itemize}
			\item $\pi$ is crepant;
			\item $\pi$ is symplectic;
			\item $K_{\widetilde{X}}$ is trivial;
			\item for every symplectic form $\omega'$ on $X^{\reg}$, its pull-back $\pi^{[*]}\omega'$ extends to a symplectic form on $\widetilde{X}$.
		\end{itemize}
	\end{prop}
	\begin{proof}
		See \cite[Proposition 1.6]{fu2005survey}.
	\end{proof}
	
	Finally, Kaledin showed in \cite{kaledin2006symplectic} that there exists a canonical stratification of a symplectic variety $X$.
	\begin{thm}
		Let $X$ be a symplectic variety. Then there exists a canonical stratification of closed subvarieties $X=X_0\supset X_1\supset \dots$ such that:
		\begin{itemize}
			\item $X_{i+1}$ is the singular locus of $X_i$;
			\item the normalization of every irreducible component of $X_i$ is a symplectic variety.
		\end{itemize}
	\end{thm}
	\begin{proof}
		See \cite[Theorem 2.3]{kaledin2006symplectic}
	\end{proof}
	We call \emph{leaf} each stratum of this decomposition. In particular each irreducible component of a leaf has even dimension.
	\section{Irreducible holomorphic symplectic manifolds}
	
	In this section we recall briefly a few facts about irreducible holomorphic symplectic manifolds. For more details see \cite{Huybrechts2003CompactHM} and \cite{debarreHKbibbia}. For any integer lattice $L$ we will denote by $L_{\mathbb{C}}:=L\otimes\mathbb{C}$ its $\mathbb{C}$-linear extension. \newline\newline
	Among the symplectic varieties it is surely notable the smooth case. Indeed, a well studied subject of modern mathematics is the geometry of \emph{irreducible holomorphic symplectic manifolds}.
	\begin{dfn}
		An irreducible holomorphic symplectic (from now on IHS) manifold $X$ is a compact complex Kähler manifold which is simply connected and such that there exists an everywhere non-degenerate holomorphic 2-form $\omega_X$ such that $H^0(X, \Omega_X^2)=\mathbb{C}\omega_X$.
	\end{dfn}
	The theory of IHS manifolds is a well studied branch of mathematics. The deformation types of IHS manifolds known at the state of art are those of the Hilbert scheme over $n$ points on a $K3$ surface and of the generalized Kummer manifolds, both due to Beauville \cite{Beauville1983} except for the case of the Hilbert square over a $K3$ surface which is due to Fujiki \cite{Fujiki}, plus other two examples constructed by O'Grady in dimension 6 \cite{OGrady6} and 10 \cite{OGrady10}.\\
	A very important feature of IHS manifolds is that the second cohomology group $H^2(X,\mathbb{Z})$ of an IHS manifold $X$ is torsion-free and it is equipped with a nondegenerate symmetric bilinear form (known as Beauville--Bogomolov--Fujiki form), which gives it the structure of an integral lattice (see \cite{Fujiki19870}). This lattice is a deformation invariant and in literature  there can be found an explicit description for each example, therefore we will say that an IHS manifold $X$ is of type $L$ if $H^2(X,\mathbb{Z})\simeq L$ implying that the deformation type is fixed.
	Denote with $\mathcal{M}_L$ the space of equivalence classes of pairs $\left(X, \eta\right)$ where $X$ is an IHS manifold of type $L$ and $\eta: H^2(X, \mathbb{Z})\to L$ an isometry of lattices, called \emph{marking}. The equivalence relation is given by: $(X, \eta)\sim (X', \eta')$ if and only if there exists a biregular morphism $\psi: X\to X'$ such that $\eta= \eta'\circ\psi^*$. This is a coarse moduli space as proven in \cite[Section 4]{Huybrechts2003CompactHM}. On this space one can define a \emph{period map} 
	\begin{align*}
		\mathcal{M}_L & \rightarrow \Omega_L \\
		(X, \eta) & \mapsto \eta(H^{2,0}(X))
	\end{align*}
	where $\Omega_L:=\left\{\omega\in\mathbb{P}(L_{\mathbb{C}})\mid (\omega,\omega)=0, \ (\omega, \bar{\omega})>0\right\}$ is called \emph{period domain}. This map is a local homeomorphism \cite[Theorem 5]{Beauville1983} and surjective when restricted to any connected component $\mathcal{M}_L^{\circ}\subset\mathcal{M}_L$ \cite[Theorem 8.1]{Huybrechts2003CompactHM}. \\
	We define the \emph{positive cone} of an IHS manifold $X$ to be the connected component $\mathcal{C}_X$ of the set $\left\{x\in H^{1,1}(X, \mathbb{R})\mid x^2>0\right\}$ containing a Kähler class and its Kähler cone $\mathcal{K}_X\subset H^{1,1}(X, \mathbb{R})$ as the cone consisting of Kähler classes.
	\newline Now we continue our review talking about automorphisms on an IHS manifold $X$. A natural way to characterize them is to look at their action on the symplectic form $\omega_X$, in fact an automorphism $\sigma$ is said \emph{symplectic} if its action is trivial on $\omega_X$, i.e. $\sigma^*(\omega_X)=\omega_X$, and non-symplectic otherwise. In this last case one can check with a simple computation that if an element is in the invariant lattice $H^2(X, \mathbb{Z})^{\sigma^*}$ then it is orthogonal to the symplectic form $\omega_X$. Therefore, $H^2(X, \mathbb{Z})^{\sigma^*}$ is contained in the Néron--Severi lattice $\NS(X)$. Moreover, by \cite[Proposition 6]{Beauvilleremarks}, an IHS manifold admitting a non-symplectic automorphism is always projective.
	\section{Cyclic cubic fourfolds and IHS manifolds}\label{sec: definizionibase}
	In this section we introduce the main object of our study, i.e. cyclic cubic fourfolds, and we explain their relation with IHS manifolds. \newline\newline
	A cubic fourfold is said \emph{cyclic} if it can be obtained as a 3:1 cyclic cover of $\mathbb{P}^4$ branched along a cubic threefold. We are interested in cyclic cubic fourfolds whose branch loci are cubic threefolds that have one isolated ADE singularity. So, with a linear change of coordinates, we can assume that the singularity is $(1:0:0:0:0)$. An equation for the fourfold $Y$ is thus given by the vanishing of the following polynomial
	\begin{equation}\label{eq:cubiche cicliche generiche}
		F(x_0:x_1:x_2:x_3:x_4:x_5) = x_0f_2(x_1,x_2,x_3,x_4) + f_3(x_1,x_2,x_3,x_4) + x_5^3,
	\end{equation}
	where $f_i$ are sufficiently generic homogeneous polynomials of degree $i$ in $\mathbb{C}[x_1,x_2,x_3,x_4]$ such that $p$ is the only singularity of $Y$.
	
	Consider now the hyperplane $H\subset\mathbb{P}^5$ defined by $\{x_0=0\}$. In this hyperplane, which we identify with $\mathbb{P}^4$ of coordinates $(x_1:x_2:x_3:x_4:x_5)$, we consider the surface $\Sigma$ given as the complete intersection of the quadric $Q$ defined by $f_2=0$ and the cubic $K$ defined by $f_3+x_5^3=0$. This surface is deeply linked to the fourfold $Y$. The varieties $Q$ and $K$ might a priori be singular, but as we want $p$ to be the only singular point of $Y$, we check that $K$ is smooth and all singular points of $\Sigma$ are given by the singular points of $Q$ by the following classical theorem.
	
	\begin{thm}[\protect{\cite[Theorem~2.1]{Wall}}]\label{thmWall}
		Let $q$ be a singular point of $\Sigma$. If both $Q$ and $K$ have a singularity in $q$, then the entire line $\overline{pq}$ connecting $p$ and $q$ is singular in $Y$. If $q$ is not a singularity of both $Q$ and $K$ and is an ADE singularity of type \textbf{T} for $Q$ or $K$, then one of the following holds:
		\begin{enumerate}[i)]
			\item $Q$ is smooth at $q$ and the cubic fourfold $Y$ has exactly two singularities, namely $p$ and $p'$, on the line $\overline{pq}$ and $p'$ is of type \textbf{T}.
			\item $Q$ is singular at $q$, the line $\overline{pq}$ meets $Y$ only in $p$ and the blow-up $\Bl_p(Y)$ of $Y$ in $p$  has a singularity of type \textbf{T} at $q$.
		\end{enumerate}
	\end{thm}
	
	Moreover, if $p$ is a singular point of $Y$ of type \textbf{T}, then by \cite[Lemma~2.1]{DR01}, the singular points of $\Sigma$ are of type \boldmath $\hat{T}$\unboldmath, as described by the following table.
	
	\begin{center}
		\begin{tabular}{|c||c|c|c|c|c|c|c|c|}
			\hline
			\textbf{T} & \boldmath $A_1$ & \boldmath $A_2$ & \boldmath $A_{n\geq 3}$ & \boldmath $D_{4}$ & \boldmath $D_{n\geq 5}$ & \boldmath $E_6$ & \boldmath $E_7$ & \boldmath $E_8$\\
			\hline
			\boldmath$\hat{T}$ & $\emptyset$ & $\emptyset$ &  \boldmath $A_{n-2}$ & 3\boldmath $A_1$ & \boldmath $A_1$ +\boldmath $D_{n-2}$ & \boldmath $A_5$ &  \boldmath $D_{6}$ & \boldmath $E_7$\\
			\hline
		\end{tabular}
	\end{center}
	Note that there exist cases in which there is more than one singular point on $\Sigma$, indeed the number of singularities of $\Sigma$ depends on the number of solutions of the equation $f_3(1, 0, 0, 0, x)=0$.
	The main reason we are interested in this surface is that it is naturally embedded in $F(Y)$, the Fano variety of lines on $Y$. Indeed, if we put together results from \cite[Lemma 3.3, Theorem 3.6]{lehn2015twisted} and \cite[Lemma 6.3.1]{hassett_2000}, we deduce the following.
	
	\begin{thm}\label{Lehn}
		Let $Y\subset\mathbb{P}^5$ be a cubic fourfold with simple isolated singularities, and suppose that it is neither reducible nor a cone over a cubic threefold. Let $p\in Y$ be a singular point and suppose that there exist no plane $\Pi\subset Y$ such that $p\in\Pi$, then the minimal resolution of $\Sigma\simeq F(Y,p)$, the Fano variety of lines in $Y$ passing through $p$, is a $K3$ surface. Moreover, the Fano variety of lines $F(Y)$ of $Y$ is birational to $\hilb^2(\Sigma)$.
	\end{thm}
	
	\begin{proof}
		We write here explicitly the map $\varphi: F(Y)\dashrightarrow \hilb^2(\Sigma) $, as it will be useful for the next sections. 
		
		By the assumptions of genericity made on $f_2$ and $f_3$, every line $l\subset Y$ passing through the singularity $p$ cuts the hyperplane $H$ in exactly one point and this defines the isomorphism between $\Sigma$ and $F(Y, p):=\{l\in F(Y) \mid  p\in l\}$. Moreover, $\Sigma$ is a (2,3)-complete intersection in $\mathbb{P}^4$ having only isolated ADE singularities, so it admits a minimal model which is a $K3$ surface.
		
		Now, consider $W\subset Y$ the cone over $\Sigma$ with vertex $p$. This is a Cartier divisor on $Y$ cut out by the equation $f_2=0$. Hence a generic line $l\subset Y$ intersects $W$ in exactly two points counted with multiplicity, thus defining a closed subscheme $\xi_{l\cap W}$ of length two on $\Sigma$. Therefore, we can define the birational map 
		\begin{align*}
			\varphi^{-1}:F(Y) & \dashrightarrow \hilb^2(\Sigma) \\
			l & \mapsto \xi_{l\cap W}
		\end{align*}
		The birational inverse of $\varphi$ is given by the natural map
		\begin{align*}
			\varphi:\hilb^2(\Sigma) & \to F(Y) \\
			\xi & \mapsto l_{\xi}
		\end{align*}
		where we define the residual line $l_{\xi}$ as follows. The intersection between $Y$ and $\langle\xi, p\rangle\simeq\mathbb{P}^2$ consists of a cone over $\xi$ and a line $l_\xi$. 
	\end{proof}
	\begin{rmk}
		The hypothesis of not having planes through the singularity is a generic assumption. There are many ways to see this, the easiest is to do a direct computation as done in \Cref{sec:planes}.
	\end{rmk}
	\begin{rmk}
		Note that $\varphi$ has no indeterminacy points. Moreover, note that the indeterminacy locus of $\varphi^{-1}$ is contained in $F(Y, p)\simeq \Sigma$. Indeed, looking at the definition of $\varphi^{-1}$  in the proof above we can see that it is not defined when a line $l\subset Y$ is contained in $W$, the cone over $\Sigma$ with vertex $p$. This means that either $l\subset \Sigma$ or $p\in l$, the former is impossible otherwise the plane $\Pi_{l,p}:=\langle l, p\rangle$ would be contained in $Y$.
	\end{rmk}
	
	\section{A symplectic resolution for $F(Y)$}\label{sec:blowuppa}
	In this section we determine the existence of a symplectic resolution for $F(Y)$ when $Y$ is a cyclic cubic fourfold branched on a cubic threefold having one isolated singularity of type $A_i$ for $i=2,3,4$. We will keep the same notation of \Cref{sec: definizionibase}.
	
	In order to do the computations, let us consider the following equation:
	\begin{equation}\begin{split}\label{eq:cubic}
			F(x_0, \dots, x_5)&= x_0Q(x_2, x_3, x_4)+K(x_1, \dots, x_5)= x_0q_1(x_2, x_3, x_4)+\\ &+x_1^2h_2(x_2, x_3, x_4, x_5)+ x_1q_2(x_2, x_3, x_4, x_5)+k_2(x_2, x_3, x_4, x_5)
		\end{split}
	\end{equation}
	with $q_1=Q$ an homogoeneous polynomial of degree 2 and $k_2$, $q_2$ and $h_2$ homogeneous polynomials which are the part of degree, respectively, three, two and one in $(x_2, ..., x_5)$ of $K(x_1, ..., x_5)$. This equation is close to the one studied by Boissière--Heckel--Sarti \cite[Section 3, Equation (3.2)]{boissiere2023fano} for the cyclic cubic fourfold branched over a cubic threefold with one singularity of type $A_1$. Here we put, following their notation, $h_1=0$ or, equivalently, we considered a rank 3 quadric given by $\{f_2=0\}$.
	\begin{rmk}
		Consider a cyclic cubic fourfold $Y$ whose branch locus is a cubic threefold with one isolated singularity of type $A_k$ and $k>1$. The equation defining such cubic can be brought to the form of \Cref{eq:cubic}.  Indeed, as noted in \Cref{sec: definizionibase}, an equation for $Y$ can be brought to the form of \Cref{eq:cubiche cicliche generiche}. To see that the rank of $f_2$ for these fourfolds is three note that, in a chart containing the singular point, it is equal to the rank of the Hessian at the origin. Consequently, by drawing a comparison with the local analytic form of a singularity of type $A_k$ (see \cite[Section 1]{Arnold}), specifically $x_1^{k+1} + x_2^2 + x_3^2 + x_4^2$, we see that this is three.
	\end{rmk}
	Now, consider $F(Y)$. This is a singular variety with singular locus $F(Y, p)$ by \cite[Corollary 1.11]{AltmanKleiman}. Moreover, the latter is isomorphic, by \Cref{Lehn}, to the singular $K3$ surface $\Sigma$. All the singularities of $\Sigma$ are in the affine chart $x_1\neq 0$, so in order to resolve its singularities we can do a local computation. The point $q_0=(0:1:0:0:0:0)\in\Sing(\Sigma)$, so we call $l_0$ the line $\overline{pq_0}$, i.e. the line corresponding to $q_0$ under the isomorphism $F(Y, p)\simeq\Sigma$. Now consider the Plücker embedding $\Gr(2,6)\hookrightarrow\Pn^{14}$, the Plücker relations yield that $\Gr(2,6)$ is locally given by eight complex coordinates. So, in an affine neighbourhood $U$ of $[l_0]$ we choose Plücker coordinates $(p_{02}, ..., p_{05}, p_{12}, ..., p_{15})$ characterizing the lines passing through the following points: 
	\begin{equation*}
		(1:0:-p_{12}:-p_{13}:-p_{14}:-p_{15}), \ \ (0:1:p_{02}:p_{03}:p_{04}:p_{05}).
	\end{equation*}
	Moreover, we put $(p_{i2},p_{i3},p_{i4})=:p_i$ for better readability. \newline
	
	On this chart a line in $\Pn^5$ is given by:\begin{equation*}
		x_0=\lambda, \ x_1=\mu, \ x_5=-\lambda p_{15}+\mu p_{05}, \ (x_2, x_3, x_4)=-\lambda p_1+\mu p_0
	\end{equation*}
	with $(\lambda:\mu)\in\Pn^1$. Then to find equations for $F(Y)$ we can substitute these expressions in \Cref{eq:cubic} and extract the homogeneous components $\phi^{i, j}$ of degree $(i, j)$ in $(\lambda, \mu)$. Then the equations become:
	\begin{align*}
		\phi^{3, 0}&=q_1(p_1)-k_2(p_1, p_{15})\\
		\phi^{2, 1}&=-2B_1(p_0, p_1)+q_2(p_1, p_{15})+k_2^{2,1}((p_0, p_{05}), (p_1, p_{15}))\\
		\phi^{1, 2}&=q_1(p_0)-h_2(p_1, p_{15})-2B_2((p_0, p_{05}), (p_1, p_{15}))-k_2^{1, 2}((p_0, p_{05}), (p_1, p_{15}))\\
		\phi^{0, 3}&=h_2(p_0, p_{05})+q_2(p_0, p_{05})+k_2(p_0, p_{05}).
	\end{align*}
	Here we denoted with $B_i$ the bilinear forms relative to $q_i$ and with $k_2^{i, j}$ the form of weight $(i,j)$ relative to $k_2$. 
	\begin{prop}\label{prop:blowup}
		The blow-up map $\alpha: \Blow{\Sigma}{F(Y)}\to F(Y)$ is a resolution of the indeterminacies of the rational map $\varphi^{-1}: F(Y)\to \hilb^{2}(\Sigma)$ mentioned in \Cref{Lehn} and $\Blow{\Sigma}{F(Y)}\simeq\hilb^{2}(\Sigma)$.
	\end{prop}
	\begin{proof}
		The argument is the same of \cite[Theorem 3.1]{boissiere2023fano}.
		First, let us recall the definition that the map $\varphi$ is a map which associates to any closed subscheme of length two $\xi$ on $\Sigma$ the residual line given by the intersection of the plane $\langle \xi, p\rangle$ and the cubic fourfold $Y$. We want to prove that the rational map $\alpha^{-1}\circ\varphi$ is a bijection and conclude that it is an isomorphism with Zariski's main theorem. To prove the bijection we will compute the fibres of $\alpha$ and $\varphi$ to show that they are isomorphic. \\ 
		First, note that we just need to check what happens over the singularities of $\Sigma$. Indeed, in order to compute the fibers over nonsingular points for the $K3$ surface we can reduce to the smooth case studied in \cite{boissiere2023fano}. This is because, if we want to study the equation locally, in a neighbourhood of a line $l_{\Bar{x}}$ corresponding to a point $\Bar{x}:=(0:\Bar{x}_1:\Bar{x}_2:\Bar{x}_3:\Bar{x}_4:\Bar{x}_5)\in\mathbb{P}^5$, then we can perform a change of variable bringing $\Bar{x}$ to the origin and do the same computations done in \textit{loc.cit.}. See \Cref{sec: translation} for the explicit transformation. \\
		Let us now compute the fibers of the morphism $\varphi$. Any plane $\Pi_a$ containing the line $l_0$ corresponding to the origin cuts $x_0=x_1=0$ in only one point of coordinates $(0:0:a_2:a_3:a_4:a_5)$ corresponding to $a:=(a_2:a_3:a_4:a_5)\in\mathbb{P}^3$. The intersection $\Pi_a\cap Y$ is a plane cubic in $\mathbb{P}^2$ of coordinates $(b_0:b_1:b_2)$ given by the equation $F(b_0:b_1:b_2a)=0$. The line $l_0$ has equation $b_2=0$ on this plane and the residual conic is given by the equation:
		\begin{equation*}
			b_0b_2q_1(a_2, a_3, a_4)+b_1^2h_2(a_2, a_3, a_4, a_5)+ b_1b_2q_2(a_2, a_3, a_4, a_5)+b_2^2k_2(a_2, a_3, a_4, a_5).
		\end{equation*}
		The fiber of $\varphi$ over $l_0$ is given by those planes whose residual conic is the union of two lines through the singular point $p$. Thus it is isomorphic to the cone $$\widetilde{C}:=\{ (a_2:a_3:a_4:a_5)\in\mathbb{P}^3\mid q_1(a_2, a_3, a_4)=0 \}.$$
		Now, in order to compute the blow-up $\Blow{\Sigma}{F(Y)}$, we compute its local expression on the chart $U$. So, it is given locally as the closure of the image of the regular morphism:
		\begin{equation*}
			\begin{tikzcd}
				U\setminus(\Sigma\cap U)\ar[r] & U\times\mathbb{P}^3\\
				((p_0, p_{05}), (p_1, p_{15}))\ar[r, mapsto] & (((p_0, p_{05}), (p_1, p_{15})), (p_{12}:p_{13}:p_{14}:p_{15})).
			\end{tikzcd}
		\end{equation*}
		Denote with $a:=(a_2:a_3:a_4:a_5)$ the coordinates of the $\mathbb{P}^3$. Assuming $a_5\neq0$ put $a_5=1$ and the relations of the blow-up become \begin{equation*}
			p_{1i}=p_{15}a_j
		\end{equation*}
		for $j=2, 3, 4$. Therefore, the equations of $\Bl_{\Sigma}(F(Y))$ on the local chart become:
		\begin{align*}
			\widetilde{\phi}^{3, 0}&=q_1(a)-p_{15}k_2(a, 1)\\
			\widetilde{\phi}^{2, 1}&=-2B_1(p_0, a)+p_{15}(q_2(a, 1)+k_2^{2,1}((p_0, p_{05}), (a, 1)))\\
			\widetilde{\phi}^{1, 2}&=q_1(p_0)-p_{15}(h_2(a, 1)+2B_2((p_0, p_{05}), (a, 1))+k_2^{1, 2}((p_0, p_{05}), (a, 1)))\\
			\widetilde{\phi}^{0, 3}&=h_2(p_0, p_{05})+q_2(p_0, p_{05})+k_2(p_0, p_{05}).
		\end{align*}
		The equation of the fiber under $\alpha$ of a line $l_0$ corresponding to the origin is found putting $p_0=0$ and $p_{15}=0$. After homogeneization: $$\alpha^{-1}(l_0)=\{(a_2:a_3:a_4:a_5)\in\mathbb{P}^3\mid q_1(a_2, a_3, a_4)=0\}=\widetilde{C}.$$ Remember that there exist no planes contained in $Y$ passing through its singular point by assumption. So, the blow-up map $\alpha: \Blow{\Sigma}{F(Y)}\to F(Y)$ is a resolution of the indeterminacies of the rational map $\varphi^{-1}$ since the coordinate $a$ of a line $l_0$ selects one plane $\Pi_a$ which cuts $Y$ in three lines: $l_0, l_1$ and $l_2$. Then $l_1$ and $l_2$ (which are not necessarily distinct) define a closed subscheme of length two on $\Sigma$ as seen in the proof of \Cref{Lehn}. Then, by Zariski's main theorem, $\Blow{\Sigma}{F(Y)}\simeq\Hilb{2}{\Sigma}$.
	\end{proof}
	Consider now the symplectic resolution $\pi:\widetilde{\Sigma}\to\Sigma$ which is just a sequence of blow-ups on the singular points. Recall that $\widetilde{\Sigma}$ is a $K3$ surface and thus we consider its Hilbert square $\hilb^2(\widetilde{\Sigma})$ which is an IHS manifold. The map $\pi$ induces a birational map $\pi^{[2]}:\hilb^2(\widetilde{\Sigma})\dashrightarrow\hilb^2(\Sigma)$ in the following way. To a generic closed subscheme $\xi\xhookrightarrow{\iota_{\xi}}\widetilde{\Sigma}$ of length 2 it associates the scheme theoretic image of $\iota_{\xi}\circ\pi$.
	\begin{rmk}
		The inverse morphism $(\pi^{[2]})^{-1}$ restrict to an isomorphims on $\hilb^2(\Sigma)\setminus Z$ where $Z:=\{\xi\in\hilb^2(\Sigma)\mid \text{supp}(\xi)\cap\Sing(\Sigma)\neq\emptyset\}$. It is easy to see that $Z$ has dimension at most 2.
	\end{rmk}
	\begin{prop}\label{prop:Samuel-Lehn}
		If there exists a birational map $f: X\dashrightarrow Y$ between a normal, irreducible, projective variety $X$ and an IHS manifold $Y$ and a closed subscheme $Z\subset X$ of codimension at least 2 such that $f_{|X\setminus Z}: X\setminus Z\to Y'\subset Y$ is an isomorphism, then $X$ is a symplectic variety with the symplectic form induced by $f$.
	\end{prop}
	\begin{proof}
		The argument is close to \cite[Theorem 3.6]{lehn2015twisted} and stated explicitly in \cite[Theorem 3.1]{boissiere2023fano}, we write it here for the sake of completeness.
		From the fact that $X'$ is isomorphic to an open subset of $Y$ we deduce that the canonical bundle is trivial on $X'$. Then as $X$ is a normal, irreducible variety and the codimension of $Z$ is at least two, we obtain $K_{X}=0$ as a Cartier divisor. We want to prove now that $X$ has symplectic singularities using \cite[Theorem 6]{Namikawa}, i.e. we need to prove that $X$ has rational Gorenstein singularities and the regular locus $X^{\reg}$ of $X$ admits an everywhere non-degenerate holomorphic closed 2-form. Let $W$ be the desingularization of an elimination of indeterminacies for $f$, so that the following diagram exists and commutes
		\[
		\xymatrix{
			&W \ar[dl]_p\ar[dr]^q  & \\
			X \ar@{-->}[rr]^f                     && Y.
		}
		\]
		Then $H^0(W, \mathcal{O}_W(K_W))=H^0(W, q^*(\mathcal{O}_Y(K_Y))=H^0(Y, \mathcal{O}_Y(K_Y))=\mathbb{C}$ since $Y$ is an IHS manifold, so $K_W$ is effective. But $K_W-p^*K_X=K_W$ as we proved that $K_X$ is trivial and thus the singularities of $X$ are canonical. This implies that $X$ has rational singularities by Elkik–Flenner theorem (\cite[Section 3, page 363]{Reid1985YoungPG}). It remains to prove that $f$ induces an everywhere non-degenerate holomorphic closed 2-form on $X^{\reg}$. Note that clearly $X^{\reg}\setminus Z\simeq X'$ as $X'$ is isomorphic to a smooth open subset of $Y$. Since $X'$ is isomorphic to an open subset of $Y$ it admits a symplectic form inherited from that of $Y'$. Now, $H^0(X', \Omega^2_{X'})\simeq H^0(X^{\reg}, \Omega^2_{X^{\reg}})$ as $\Omega^2_{X^{\reg}}$ is reflexive and by \cite[Theorem 1.6]{Hartshornereflsheaves} any reflexive sheaf is normal . So any symplectic 2-form on $X'$ extends to the whole $X^{\reg}$ and it is still closed as $X$ has rational singularities (see \cite[Theorem 1.13]{KebekusSchnell}). Moreover, it is also non-degenerate otherwise it would degenerate along a divisor cutting also $X'$. Therefore by \cite[Theorem 6]{Namikawa} we deduce that $X$ has at most symplectic singularities and thus it is by definition a symplectic variety. 
	\end{proof}
	\begin{corol}
		The varieties $\hilb^2({\Sigma})$ and $F(Y)$ are symplectic varieties which admit a symplectic resolution.
	\end{corol}
	\begin{proof}
		Note that both $(\pi^{[2]})^{-1}$ and $(\pi^{[2]})^{-1}\circ\varphi^{-1}$ respect the hypotheses of \Cref{prop:Samuel-Lehn}. Moreover, they admit a symplectic resolution by \cite[Corollary 5.6]{lehn2015twisted} and \cite[Proposition 3.5]{YamagishiFano}.
	\end{proof}
	\begin{rmk}
		The fact that both the varieties $\hilb^2({\Sigma})$ and $F(Y)$ admit a symplectic resolution can be proven with a more general approach, see \cite[Remark 5.4]{KLSV} for the details.
	\end{rmk}
	We will call $\widetilde{\hilb^2(\Sigma)}$ a symplectic resolution of $\hilb^2(\Sigma)$ and, thus, of $F(Y)$ by \Cref{prop:blowup}. This does not imply that every symplectic resolution $\widetilde{F(Y)}$ of $F(Y)$ is of this form. Indeed, a priori, it is not true that any symplectic resolution of a variety factors through its blow-up on the singular locus, but using the fact that $F(Y)$ is a four-dimensional variety we can prove the following lemma.
	\begin{lemma}\label{lemma: risoluzioni simplettiche}
		Every symplectic resolution $R$ of $F(Y)$ factors through the blow-up $\Blow{\Sigma}{F(Y)}$.
	\end{lemma}
	\begin{proof}
		Consider the symplectic resolution:
		\begin{equation*}
			\gamma:\widetilde{\hilb^2(\Sigma)} \to  F(Y)
		\end{equation*}
		By \cite[Theorem 1.1]{wierzba2002small} the exceptional locus $E$ of $\gamma$ can be either a divisor or 2-dimensional. In the latter case $\dim(\gamma(E))=0$ by \cite[Lemma 2.1]{wierzba2002small}. Therefore, as the singular locus $\Sigma$ of $F(Y)$ is a surface, the map $\gamma$ contracts a divisor $E$ into $\Sigma$. Moreover, as $\widetilde{\hilb^2(\Sigma)}$ is smooth the divisor $E$ is Cartier and by the universal property of blow-up there exists a unique map $\gamma':\widetilde{\hilb^2(\Sigma)}\to\Blow{\Sigma}{F(Y)}$ which factors $\gamma$ through the blow-up.
	\end{proof}
	Thus, we proved that every symplectic resolution of $F(Y)$ is isomorphic to a symplectic resolution of $\hilb^2(\Sigma)$. Nevertheless, in order to highlight the point of view which we are using we will also denote by $\widetilde{F(Y)}$ a symplectic resolution of $F(Y)$.
	
	\section{Geometry of $F(Y)$}
	
	In this section we investigate some geometric properties of $F(Y)$ when $Y$ is a cyclic fourfold branched on a cubic threefold having one isolated singularity of type $A_i$ for $i=2,3,4$. \\
	
	First, we want to study the nature of the singular points of $F(Y)$ on the 2-dimensional leaf, i.e. $$\Sing(F(Y))\setminus\Sing(\Sing(F(Y)))\simeq \Sigma\setminus\{\Sing(\Sigma)\}.$$ 
	\begin{prop}\label{prop:descrizione locale}
		For every point of $\Sigma\setminus\{\Sing(\Sigma)\}\subset F(Y)$ there exists a neighbourhood of $F(Y)$ which is analytically isomorphic to $(\mathbb{C}^2, 0)\times (\Gamma, t)$ with $(\Gamma, t)$ the germ of a point on a surface having an isolated singularity on t of type:
		\begin{enumerate}[i)]
			\item $D_4$ if $C$ has an isolated $A_2$ singularity;
			\item $E_6$ if $C$ has an isolated $A_3$ singularity;
			\item $E_8$ if $C$ has an isolated $A_4$ singularity.
		\end{enumerate}
	\end{prop}
	\begin{proof}
		In order to do local computations we want to change the local chart given by Plücker coordinates given in \Cref{sec:blowuppa} in a way that a point of the 2-dimensional leaf, which we can assume to be $(0:0:1:0:0:0)$, is the origin in the new chart. Therefore, we choose Plücker coordinates characterizing the lines passing through the points: 
		\begin{equation*}
			(1:-p_{11}:0:-p_{13}:-p_{14}:-p_{15}), \ \ (0:p_{01}:0:p_{03}:p_{04}:p_{05}).
		\end{equation*}
		Moreover we may assume (with a slight abuse of notation) that after a linear coordinate change $q_1(x_2, x_3, x_4)=x_2h_1(x_3, x_4)+q_1(x_3, x_4)$ and $K(x_1, \dots, x_5)=x_2^2h_3(x_1, x_3, x_4, x_5)+ x_2q_3(x_1, x_3, x_4, x_5)+k_2(x_1, x_3, x_4, x_5)$.
		With computations analogous to \Cref{sec:blowuppa} we get the following equations for $F(Y)$:
		\begin{align*}
			\phi^{3, 0}&=q_1(\Bar{p_1})-k_2(p_{11}, \Bar{p_1}, p_{15})\\
			\phi^{2, 1}&=-2B_1(\Bar{p_0}, \Bar{p_1})+q_3(p_{11}, \Bar{p_1}, p_{15})+k_2^{2,1}(p_{01}, \Bar{p_0}, p_{05}, p_{11}, \Bar{p_1}, p_{15})-h_1(\Bar{p_1})\\
			\phi^{1, 2}&=q_1(\Bar{p_0})+h_1(\Bar{p_0})-h_3(p_{11}, \Bar{p_1}, p_{15})-2B_2(p_{01}, \Bar{p_0}, p_{05}, p_{11}, \Bar{p_1}, p_{15})+k_2^{1, 2}(p_{01}, \Bar{p_0}, p_{05}, p_{11}, \Bar{p_1}, p_{15})\\
			\phi^{0, 3}&=h_3(p_{01},\Bar{p_0}, p_{05})+q_3(p_{01},\Bar{p_0}, p_{05})+k_2(p_{01},\Bar{p_0}, p_{05}).
		\end{align*}
		Here we put $\Bar{p_i}=(p_{i3}, p_{i4})$.
		In order to determine the nature of the singularity at the origin we use the same argument of \cite[Theorem 4.1 item (2)]{boissiere2023fano}. Note that in a neighbourhood of a nonsingular point of $\Sigma$ the hypersurfaces $Q:=\{Q=0\}$ and $K:=\{K=0\}$ meet transversally, therefore $h_1$ and $h_3$ are not proportional. Hence we can suppose that after a linear change of coordinates $h_1(x_3, x_4)=x_3$ and $h_3(x_1, x_3, x_4, x_5)=x_4$. Therefore, we can use the equations $\phi^{1, 2}$ and $\phi^{0, 3}$ to obtain complex analytic local expressions $\widetilde{p_{03}}$, $\widetilde{p_{04}}$ for $p_{03}$ and $p_{04}$ in terms of $p_{01}, p_{05}, p_{11}, p_{13},p_{14}, p_{15}$. Thus, there exists a local biholomorphism in a neighbourhood of the origin between our variety and the variety $\chi\subset \mathbb{C}^2\times\mathbb{C}^4$ described by the equations:
		\begin{align*}
			\overline{\phi^{3, 0}}&=q_1(\Bar{p_1})-k_2(p_{11}, \Bar{p_1}, p_{15})\\
			\overline{\phi^{2, 1}}&=-2B_1(\widetilde{p_{03}}, \widetilde{p_{04}}, \Bar{p_1})+q_3(p_{11}, \Bar{p_1}, p_{15})+k_2^{2,1}(p_{01}, \widetilde{p_{03}}, \widetilde{p_{04}}, p_{05}, p_{11}, \Bar{p_1}, p_{15})-h_1(\Bar{p_1}).
		\end{align*}
		Here $(p_{01}, p_{05})$ are local coordinates for $\Sigma$. By \cite[Theorem 2.3]{kaledin2006symplectic} (see also \cite[Proposition 2.2]{lehn2021deformations}) we know that every point in $\Sigma\setminus\Sing(\Sigma)$ has a 
		neighbourhood which is locally analytically isomorphic to $(\mathbb{C}^2, 0)\times (\Gamma, t)$ with $(\Gamma, t)$ the germ of a smooth point or a rational double point on a surface. Therefore, as we want to understand the structure over the origin, we put $p_{01}=p_{05}=0$. Thus we can consider the surface $\Gamma$ given by the following equations in $\mathbb{C}^4$:
		\begin{align*}
			\overline{\phi^{3, 0}}&=q_1(\Bar{p_1})-k_2(p_{11}, \Bar{p_1}, p_{15})\\
			\overline{\phi^{2, 1}}&=q_3(p_{11}, \Bar{p_1}, p_{15})-h_1(\Bar{p_1}).
		\end{align*}
		Again we can consider $h_1(\Bar{p_1})=p_{13}$ and use a local inversion of the second equation to obtain a locally $\widetilde{p_{13}}$ for $p_{13}$ as a quadratic expression in $p_{11}, p_{14}, p_{15}$. Now to draw conclusions we need to specialize our manifold to our case. Let us begin with putting the condition of being a cyclic cubic fourfold, then an equation for $\Gamma$ in $\mathbb{C}^3$ given by $p_{11}, p_{14}, p_{15}$ becomes:
		\begin{equation*}
			\overline{\phi^{3, 0}}=q_1(\widetilde{p_{13}}, p_{14})-k_2(p_{11}, \widetilde{p_{13}}, p_{14})+ p_{15}^3.
		\end{equation*}
		Then the following things can happen (see \Cref{sec:computation of descr locale} for explicit computations):
		\begin{enumerate}[i)]
			\item $p_{11}^3$ appears in $k_2(p_{11}, \widetilde{p_{13}}, p_{14})$: then the equation is semiquasihomogenous (SQH) of degree $(\frac{1}{3}, \frac{1}{2}, \frac{1}{3})$ thus it has a $D_4$ singularity at the origin. The condition on $K$ is equivalent to ask an isolated $A_2$ singularity for the cubic threefold $C$.
			\item $p_{11}^4$ appears in $\overline{\phi^{3, 0}}$ and $p_{11}^3$ does not appear in $k_2(p_{11}, \widetilde{p_{13}}, p_{14})$: this monomial appears thanks to both $q_1$ and $k_2$, so if it is not eliminated then the equation is SQH of degree $(\frac{1}{4}, \frac{1}{2}, \frac{1}{3})$ yielding an $E_6$ singularity. This condition is obtained when $C$ has an isolated $A_3$ singularity.
			\item $p_{11}^5$ appears in $k_2(p_{11}, \widetilde{p_{13}}, p_{14})$, $p_{11}^4$ does not appear in $\overline{\phi^{3, 0}}$ and $p_{11}^3$ does not appear in $k_2(p_{11}, \widetilde{p_{13}}, p_{14})$: the equation is SQH of degree $(\frac{1}{5}, \frac{1}{2}, \frac{1}{3})$, yielding an $E_8$ singularity. This condition is satisfied when $C$ has an isolated $A_4$ singularity.
		\end{enumerate}
		Therefore $(\Gamma, 0)$ is determined by the above calculations.
	\end{proof}
	As also noted in \cite{libassi2023git} the behaviour of $F(Y)$ when $Y$ is branched on a cubic fourfold of type $A_2$ is very different from the $A_3$ and $A_4$ cases, therefore, we will divide the study in two different sections.

	\subsection{Cubic fourfold branched on a threefold with an $A_3$ or $A_4$ singularity}\label{sec: A3A4} \hfill\\\\
	We start by studying the geometry of $F(Y_i)$ in the case where $Y_i$ is branched on a cubic threefold having one isolated singularity of type $A_3$ or $A_4$.\\
	
	Consider the covering automorphism $\sigma$ defined on $Y_i$. This automorphism is just the identity on the first 5 coordinates and maps $x_5\mapsto\xi_3\cdot x_5$ with $\xi_3$ a third primitive root of the unity. It is a projectivity and, thus, maps lines to lines. Therefore it induces a linear automorphism on $F(Y_i)$. This is a linear automorphism mapping the singular locus to itself, therefore there exists only one automorphism on $\Bl_{\Sigma}F(Y_i)$ commuting with the blow-up morphism by the universal property of blow-ups.  This automorphism corresponds via the isomorphism of \Cref{prop:blowup} to the natural automorphism $\sigma_i^{[2]}$ induced on $\hilb^2(\Sigma_i)$ by the action of $\sigma$ on $\Sigma_i$. By \cite[Section 2]{yamagishi2018symplectic} if $\Sigma_i$ has an isolated singularity then $\widetilde{\hilb^2(\Sigma)}$ is obtained by a sequence of blow-ups along the successive singular loci, therefore we can iterate the above argument and induce an automorphism $\widetilde{\sigma_i^{[2]}}$ on $\widetilde{\hilb^2(\Sigma_i)}$ (see \cref{sec:auto} for the details). In \cite[Section 7.2, Section 7.3]{libassi2023git} the author studies the manifolds obtained considering the Hilbert square $\hilb^2(\widetilde{\Sigma_i})$, where $\widetilde{\Sigma_i}$ is the $K3$ surface which is the minimal resolution of $\Sigma_i$, and proves that on these manifolds there exists a non-symplectic automorphism of order three whose action in cohomology is represented by $\rho_i\in O(L)$. To be more precise, let $\rho\in O(L)$ be the automorphism induced in cohomology by $\sigma$,  $T_i:=\Pic(\Sigma_i)$ and $S_i:=T_i^{\perp}$; then  $\rho_i\in O(L)$ is the lifting of $\id_{T_i}\oplus\rho|_{S_i}$ to the whole $L$ whose existence is proven in \cite[Section 7]{libassi2023git}.
	\begin{prop}\label{prop: induction of automorphism}
		In the hypotheses of \Cref{teoremone}, the automorphism $\widetilde{\sigma^{[2]}}$ induces $\rho$ in cohomology, i.e. there exists a marking $\eta$ on $\widetilde{\hilb^2(\Sigma)}$ such that $\rho=\eta^{-1}\circ(\widetilde{\sigma^{[2]}})^*\circ\eta$.
	\end{prop}
	\begin{proof}
		The only irreducible component $E_1$ of the exceptional divisor of the first blow-up $\alpha$ is clearly preserved by the induced automorphism. If $C$ had one isolated singularity of, respectively, type $A_3$ or $A_4$ then $\Sigma$ has, respectively:
		\begin{itemize}
			\item one isolated singularity of type $A_5$. Then by \cite[Section 2]{yamagishi2018symplectic} $\widetilde{\hilb^2(\Sigma)}$ is obtained by $\hilb^2(\Sigma)$ via 3 blow-ups, two of them introducing each one two irreducible components of the effective divisor and the third another one. At every blow-up, the induced automorphism maps the subgroup of the Picard group generated by the irreducible components of the exceptional divisor into itself because it commutes with the composition of blow-ups by the universal property of blow-ups. As the automorphism has order three it cannot swap two irreducible components, thus it preserves all the 6 irreducible components of the exceptional divisor introduced at each blow-up.
			\item One isolated singularity of type $E_7$. Then by \cite[Section 2]{yamagishi2018symplectic} the successive blow-ups introduce on $\widetilde{\hilb^2(\Sigma)}$ 7 irreducible exceptional divisors in an $E_7$ configuration. The induced automorphism maps the subgroup of the Picard group generated by these divisors into itself as it commutes with the composition of blow-ups by the universal property of blow-ups. We recall here the picture of the blow-ups needed.
			
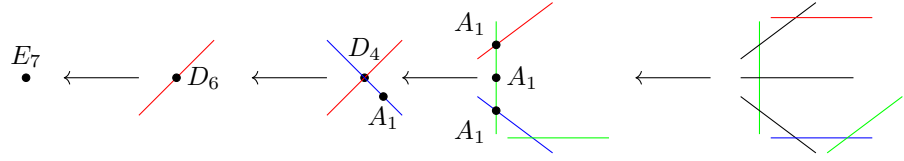
\begin{center}
				\begin{tikzpicture}
					\draw[<-]        (0.5,1)   -- (1.5,1);
					\filldraw[black] (0, 1) circle (0.5 mm) node[anchor=south]{\small{$E_7$}};
					\draw[red] (1.5,0.5) -- (2.5,1.5);
					\filldraw[black] (2, 1) circle (0.5 mm) node[anchor=west]{\small{$D_6$}};
					\draw[<-]        (3,1)   -- (4,1);
					\filldraw[black] (4.5, 1) circle (0.5 mm) node[above=3pt]{\small{$D_4$}};
					\draw[red] (4,0.5) -- (5, 1.5);
					\draw[blue] (4, 1.5) -- (5, 0.5);
					\filldraw[black] (4.75, 0.75) circle (0.5 mm) node[ below ]{\small{$A_1$}};
					\draw[<-]        (5,1)   -- (6,1);
					\draw[green] (6.25,0.25) -- (6.25, 1.75);
					\draw[red] (6, 1.25) -- (7,2);
					\draw[blue] (6, 0.75) -- (7, 0);
					\draw[green] (6.4, 0.2) -- (7.75, 0.2);
					\filldraw[black] (6.25, 0.5625) circle (0.5 mm) node[ below left ]{\small{$A_1$}};
					\filldraw[black] (6.25, 1) circle (0.5 mm) node[ right ]{\small{$A_1$}};
					\filldraw[black] (6.25, 1.4375) circle (0.5 mm) node[ above left ]{\small{$A_1$}};
					\draw[<-]        (8.1,1)   -- (9.1,1);
					\draw[green] (9.75,0.25) -- (9.75, 1.75);
					\draw (9.5, 1.25) -- (10.5,2);
					\draw (9.5, 0.75) -- (10.5, 0);
					\draw[blue] (9.9, 0.2) -- (11.25, 0.2);
					\draw[green] (10.65, 0)--(11.65,0.75);
					\draw[red] (9.9, 1.8) -- (11.25, 1.8);
					\draw (9.5, 1) -- (11, 1);
					
				\end{tikzpicture}
			\end{center}
			Each arrow represents a blow-up. Each line represents an irreducible component of the exceptional divisor which are coloured with a different colour at each blow-up. So, with the same argument of above we deduce that the irreducible components introduced at the first three blow-ups are preserved. Finally, the three irreducible components introduced by the last blow-up cannot be permuted as the intersection between different irreducible components needs to be preserved (as the first 4 irreducible components are preserved).
		\end{itemize} 
		Therefore the invariant lattice has, respectively, rank at least 7 or 9. 
		By \cite[Section 7.2, Section 7.3]{libassi2023git}, the Picard lattice is isomorphic, respectively, to $\langle 6\rangle\oplus E_6$ or $\langle 6\rangle\oplus E_8$ which have, respectively, rank 7 or 9. By \cite[Corollary 5.7]{BCS_Kyoto}, the action of natural automorphisms on IHS manifolds of $K3^{[2]}$-type are uniquely determined by the action on the invariant lattice and the possibilities are listed in \cite[Table 1]{BCS_Kyoto}. Thus, checking all the possibilities, the invariant lattices must be isomorphic to the Picard lattices and the action of $\widetilde{\sigma_i^{[2]}}$ induces $\rho_i$ in cohomology.
	\end{proof}
	When a cubic threefold $C$ has an isolated singularity of type $A_3$ or $A_4$ then the cyclic cubic fourfold $Y$ associated to $C$ has, respectively, an isolated singularity of type $E_6$ or $E_8$ by Theorem \ref{thmWall}.
	
	Therefore, considered also the results of \Cref{sec:blowuppa}, we can state the following theorem.
	\begin{thm}\label{teoremone}
		Let $C_i$ be a complex projective cubic threefold having one isolated singularity of type $A_i$ for $i= 3, 4$ and let $Y_i$ be its associated cyclic cubic fourfold. Assume that there exist no plane $\Pi\subset Y$ such that $\Pi\cap\Sing(Y_i)\neq\emptyset$. Then the Fano variety of lines $F(Y_i)$ of $Y$ admits a unique symplectic resolution by an IHS manifold of $K3^{[2]}$-type $\widetilde{F(Y_i)}$. \\
		Moreover, there exists an integral lattice $T_i$, defined below, such that:\begin{enumerate}[i)]
			\item $\Pic\left(\widetilde{F(Y_i)}\right)\simeq T_i$;
			\item there exists a non-symplectic automorphism $\tau\in\Aut\left(\widetilde{F(Y_i)}\right)$ whose invariant sublattice is $H^2(\widetilde{F(Y_i)}, \mathbb{Z})^{\tau^*}\simeq T_i$
		\end{enumerate}
		with $T_i$ and $R_i$ defined in the following table:
		\begin{center}
			\begin{tabular}{|c|c|} 
				\hline
				$i$ & $T_i$\\
				\hline
				$3$ &  $\langle 6\rangle\oplus E_6$\\ 
				\hline
				$4$ & $\langle 6\rangle\oplus E_8$\\ 
				\hline
			\end{tabular}
		\end{center}
	\end{thm}
	\begin{proof}
		First, note that by \cite{yamagishi2018symplectic}, the symplectic resolution $\widetilde{\hilb^2(\Sigma)}$ of  $\hilb^2(\Sigma)$ is unique when $\Sigma$ has one isolated singularity of type ADE. \\ \\
		Let us consider now the following diagram
		\begin{equation*}
			\begin{tikzcd}
				\hilb^2(\widetilde{\Sigma}) \ar[dr, dashed, "\pi^{[2]}"] \ar[d] & \widetilde{\hilb^2(\Sigma)}\ar[d, "\psi"] \ar[r, "\simeq"] & \widetilde{F(Y)} \ar[d]\\ \Sym^2(\widetilde{\Sigma})\ar[dr]
				& \hilb^2(\Sigma) \ar[rd, "\varphi"] \ar[r, "\simeq"] \ar[d, "hc"] & \Bl_{\Sigma}F(Y)\ar[d, "\alpha"]\\
				&\Sym^2(\Sigma)&F(Y) 
			\end{tikzcd} 
		\end{equation*}
		The left side of this diagram shows that $\hilb^2(\widetilde{\Sigma})$ and $\widetilde{\hilb^2(\Sigma)}$ are both symplectic resolutions of the same object $\Sym^2(\Sigma)$; thus, by \cite[Theorem 1.2]{wierzba2002small}, there exists a sequence of Mukai flops mapping $\hilb^2(\widetilde{\Sigma})$ to $\widetilde{\hilb^2(\Sigma)}$. This in particular means that  $\hilb^2(\widetilde{\Sigma})$, $\widetilde{\hilb^2(\Sigma)}$ and $\widetilde{F(Y)}$ have the same period as IHS manifolds of $K3^{[2]}$-type. The manifolds $\hilb^2(\widetilde{\Sigma})$ and the presence of a non-symplectic automorphism of order 3 have been deeply studied in \cite[Section 7.1, Section 7.2]{libassi2023git}. The existence of the automorphism $\tau_i$ and its action in cohomology is determined by \Cref{prop: induction of automorphism}. The details of this automorphism follow from the study made in \cite{libassi2023git}.
	\end{proof}
	Moreover, we can say more about the geometry of $\widetilde{F(Y)}$.
	\begin{prop}
		The variety $\mathcal{H}$, obtained by $F(Y)$ after a suitable number of successive blow-ups on the singular loci, has transversal ADE singularities, thus its blow-up on $\Sing(\mathcal{H})$ is a crepant resolution.
	\end{prop}
	\begin{proof}
		In \cite{yamagishi2018symplectic} the author shows with his computations that if $\Sigma$ has an isolated singularity of type $T_n$ with $T_n$ a singularity of type ADE then $\Bl_{\Sing(\hilb^2(\Sigma))}\hilb^2(\Sigma)$ has the same singularities of $\hilb^2(\Gamma)$ with $\Gamma$ a surface with an isolated singularity of type $T'_m$ and $m<n$. In particular after a suitable number of blow-ups the variety $\mathcal{H}$, obtained by successive blow-ups along singular loci, will not have $0$-dimensional symplectic leaves, i.e. $\Sing(\mathcal{H})$ will be smooth. So, by \Cref{prop:descrizione locale}, the variety $\mathcal{H}$ has only transversal singularities. The blow-up is then a crepant resolution by \cite[Proposition 4.2]{perroni2007chenruan}.
	\end{proof}
	\subsection{Cubic fourfold branched on a threefold with a $A_2$ singularity}\hfill\\\\
	Now we want to focus to the case where $C$ has an isolated singularity of type $A_2$ and thus $\Sigma$ has three $A_1$ singularities, namely $q_0, q_1$ and $q_2$. \newline \newline
	First, we want to describe the singular locus $\Sing(\hilb^2(\Sigma))$.
	\begin{prop}\label{prop: sing A2}
		Suppose that $\Sigma$ has three $A_1$ singularities $q_0, q_1$ and $q_2$, then $\Sing(\hilb^2(\Sigma))$ consists of three irreducible components $\widetilde{\Sigma_i}\simeq\Bl_{q_i}\Sigma$. An irreducible component $\widetilde{\Sigma_i}$ intersects the other two components in $q_j+q_i$, with $j\neq i$.
	\end{prop}
	\begin{proof}
		Consider the Hilbert--Chow morphism $hc:\hilb^2(\Sigma)\to\Sym^2(\Sigma)$. This morphism can be identified with the blow-up along the diagonal $\Delta$, so $$hc^{-1}(\Sing(\Sym^2(\Sigma))\setminus\Delta)\subset\Sing(\hilb^2(\Sigma))\subset hc^{-1}(\Sing(\Sym^2(\Sigma))).$$
		Now, $\Sing(\Sym^2(\Sigma))\setminus \Delta$ consists of cycles where at least one point lies in $\Sing(\Sigma)$, therefore all the length two subschemes $\xi$ such that their support consists of two different points on $\Sigma$ and at least one of them is in $\Sing(\Sigma)$ are also singular points for $\hilb^2(\Sigma)$. Therefore, the remaining points of $\hilb^2(\Sigma)$ which might possibly be singular are those on the fibers over $2q_i$ for $i=0,1,2$. These points are length 2 closed subschemes of $\Sigma$ entirely supported on an isolated singularity of type $A_1$. Therefore, using the computations about $\hilb^2(\Gamma)$ with $\Gamma$ a surface having one isolated singularity of type $A_i$ done in \cite[Section 2.1]{yamagishi2018symplectic}, we can see that $hc^{-1}(2q_i)\cap \Sing(\hilb^2(\Sigma))$ is isomorphic to the exceptional divisor $L_i$ of $\Bl_{q_i}\Sigma$. Thus, the embedding:
		\begin{align*}
			\Sigma & \to \Sing(\Sym^2(\Sigma)) \\
			p & \mapsto p+q_i
		\end{align*}
		induces an embedding:
		\begin{align*}
			\Bl_{q_i}\Sigma & \to \Sing(\hilb^2(\Sigma)) \\
			p\not\in L_i & \mapsto p+q_i \\
			p\in L_i & \mapsto p.
		\end{align*}
		Finally, the point $q_i+q_j$ with $i\neq j$ is a point which is mapped through the two different isomorphisms to $q_i+q_j\in \widetilde{\Sigma_i}$ and $q_j+q_i\in\widetilde{\Sigma_j}$.
	\end{proof}
	\begin{rmk}
		This proposition in particular implies that $\Sing(\hilb^2(\Sigma))$ has three singular points, namely $q_0+q_1$, $q_0+q_2$ and $q_1+q_2$. Indeed, the three connected components $\widetilde{\Sigma_i}\simeq\Bl_{q_i}(\Sigma)$ are smooth on the preimage of $q_i$ under the blow-up as $q_i$ is a singularity of type $A_1$ for $\Sigma$.
	\end{rmk}
	Now, we want to prove that the symplectic resolution $\widetilde{\hilb^2(\Sigma)}\xrightarrow{\psi}\hilb^2(\Sigma)$ is unique.
	\begin{prop}\label{prop: unicity A2}
		The symplectic resolution $\widetilde{\hilb^2(\Sigma)}$ of $\hilb^2(\Sigma)$ is unique up to isomorphism.
	\end{prop}
	\begin{proof}
		We want to use again \cite[Theorem 1.2]{wierzba2002small} and prove that the central fiber does not contain components isomorphic to $\mathbb{P}^2$.
		Recall that, as they are both symplectic resolutions of $\Sym^2(\Sigma)$, the variety $\widetilde{\hilb^2(\Sigma)}$ is obtained by $\hilb^2(\widetilde{\Sigma})$ via a sequence of Mukai flops $\mu:\hilb^2(\widetilde{\Sigma})\dashrightarrow\widetilde{\hilb^2(\Sigma)}$ performed over the 2-dimensional fibers of $\hilb^2(\widetilde{\Sigma})\to\Sym^2(\Sigma)$. Moreover, the birational map $\pi^{[2]}$ restricts to an isomorphism between $$\hilb^2(\widetilde{\Sigma})\setminus(\bigcup\Pi_i\cup\bigcup\Lambda_{ij})$$ and $$\hilb^2(\Sigma)\setminus (\bigcup hc^{-1}(2q_i)\cup\bigcup hc^{-1}(q_i+q_j))$$ with $i\neq j$. Here we denoted by $\Pi_i\simeq \mathbb{P}^2$ and  $\Lambda_{ij}\simeq\mathbb{P}^1\times\mathbb{P}^1$  the subspaces of $\hilb^2(\widetilde{\Sigma})$ which are, respectively, $hc^{-1}(2L_i)$ and $hc^{-1}(L_i+L_j)$. 
		The central fiber is $\psi^{-1}(q_i+q_j)\simeq\Lambda_{ij}\simeq\mathbb{P}^1\times\mathbb{P}^1$ and, thus, it does not contain any component isomorphic to $\mathbb{P}^2$.
	\end{proof}
	In analogy with \Cref{sec: A3A4} we want to prove that $\widetilde{\hilb^2(\Sigma)}$ is obtained by $\hilb^2(\Sigma)$ via successive blow-ups on the singular loci. Let $S:=\Sing(\hilb^2(\Sigma))$. 
	\begin{lemma}
		The blow-up map $\beta:\Bl_S\hilb^2(\Sigma)\to\hilb^2(\Sigma)$ is crepant.
	\end{lemma}
	\begin{proof}
		As the canonical bundle of $\hilb^2(\Sigma)$ is trivial the statement is equivalent to ask that $X:=\Bl_S\hilb^2(\Sigma)$ has trivial canonical bundle. First, note that $$\dim(\beta^{-1}(q_i+q_j))=\text{codim}(\beta^{-1}(q_i+q_j))=2$$ so if we prove that $K_X$ is trivial on $X\setminus \bigcup\beta^{-1}(q_i+q_j) $ then as $X$ is normal and irreducible we get $K_X=0$. By \Cref{prop:descrizione locale} and \Cref{prop: sing A2}, $\hilb^2(\Sigma)$ at each point of $\widetilde{\Sigma_i}\setminus{q_i+q_j}$ admits a local description as $\mathbb{C}^2\times \Gamma$ with $\Gamma$ a surface with a singularity of type $A_1$. Therefore, as shown in \cite[Proposition 4.2]{perroni2007chenruan}, the blow-up is locally isomorphic to $\widetilde{\Gamma}\times\mathbb{C}^2$,  where $\widetilde{\Gamma}$ denotes the minimal resolution of $\Gamma$, which is crepant as $\widetilde{\Gamma}\to\Gamma$ is so. 
	\end{proof}
	We can now prove the following proposition.
	\begin{lemma}\label{lemma:symplectic A2}
		$\Bl_S\hilb^2(\Sigma)$ is a symplectic variety.
	\end{lemma}
	\begin{proof}
		Consider the symplectic resolution $\psi:\widetilde{\hilb^2(\Sigma)}\to\hilb^2(\Sigma)$ and the composition of birational maps $f:=\psi^{-1}\circ\beta:\Bl_S\hilb^2(\Sigma)\dashrightarrow\widetilde{\hilb^2(\Sigma)}$. As $\beta$ is crepant  the map $f$ is defined and injective on a complement to a closed subset $Z\subset \Bl_S\hilb^2(\Sigma)$ of codimension $\text{codim}(Z)\geq 2$ by \cite[Lemma 2.3 (i)]{kaledin2001symplectic}. Then $\Bl_S\hilb^2(\Sigma)$ has only symplectic singularities follows from \Cref{prop:Samuel-Lehn}.
	\end{proof}
	Indeed we can say more. As it turns out $\Bl_S\hilb^2(\Sigma)$ is smooth and it is the symplectic resolution of $\widetilde{\hilb^2(\Sigma)}$.
	\begin{prop}
		The symplectic resolution $\widetilde{\hilb^2(\Sigma)}$ of $\hilb^2(\Sigma)$ is isomorphic to $\Bl_S\hilb^2(\Sigma)$.
	\end{prop}
	\begin{proof}
		In order to prove that $\psi:\widetilde{\hilb^2(\Sigma)}\to \hilb^2(\Sigma)$ factors through $\Bl_S\hilb^2(\Sigma)$ we want to use the same argument of \ref{lemma: risoluzioni simplettiche}. With the same argument the exceptional locus of $\psi$ must be a Cartier divisor contracted to the singular locus $S$ of $\hilb^2(\Sigma)$. Then, by the universal property of the blow-up the map $\psi$ factors through $\beta:\Bl_S\hilb^2(\Sigma)\to\hilb^2(\Sigma)$ via a map $\psi':\widetilde{\hilb^2(\Sigma)}\to\Bl_S\hilb^2(\Sigma)$. The Picard group of $\widetilde{\hilb^2(\Sigma)}$ is isomorphic to the Picard group of $\hilb^2(\widetilde{\Sigma})$ as they are two symplectic resolutions of the same symplectic variety $\Sym^2(\Sigma)$. Moreover, by \cite[Section 7.4]{libassi2023git}, it is $\Pic(\widetilde{\hilb^2(\Sigma)})\simeq D_4(-1)\oplus\langle 6\rangle$. Therefore, $\psi:\widetilde{\hilb^2(\Sigma)}\to\hilb^2(\Sigma)$ has relative Picard number $3$. By \Cref{prop: sing A2} we deduce that  $\beta:\Bl_S\hilb^2(\Sigma)\to\hilb^2(\Sigma)$ has at least relative Picard number $3$, but, as proven before, the resolution $\psi$ factors through $\beta$, thus, by confrontation of the relative Picard groups we deduce that $\psi'$ is a small symplectic contraction (see \cite[Definition 2]{wierzba2002small}). Therefore, by \cite[Theorem 1.1]{wierzba2002small}, if it is not an isomorphism it can be either a sequence of Mukai flops or a contraction of some planes. Both cases are impossible as $\psi$ factors through $\psi'$ and in the 2-dimensional fibres of $\psi$ there exist no planes. 
	\end{proof}
	Now, we are interested in the presence of an automorphism on $\widetilde{\hilb^2(\Sigma)}$. Indeed, with the same argument of \Cref{sec: A3A4}, we can induce an automorphism $\widetilde{\sigma^{[2]}}$ on $\widetilde{\hilb^2(\Sigma)}$. We are interested now in the action of $\widetilde{\sigma^{[2]}}$ in cohomology. First note that by \cite[Proposition 6.3]{libassi2023git} the automorphism $\sigma$ on $\Sigma$ induces also a natural automorphism $\tau_2$ on $\hilb^2(\widetilde{\Sigma})$, where by $\widetilde{\Sigma}$ we denote the minimal resolution of $\Sigma$. As both $\sigma^{[2]}$ and $\tau_2$ are natural automorphisms induced by $\sigma$ we can see that $\pi^{[2]}\circ\tau_2=\sigma^{[2]}\circ\pi^{[2]}$ (see \Cref{sec:blowuppa} for the definition of $\pi^{[2]}$).
	\begin{prop}\label{prop: automorphismA2}
		The automorphism $\widetilde{\sigma^{[2]}}$ has the same action of $\tau_2$ in cohomology.
	\end{prop}
	\begin{proof}
		In order to prove this it is enough to show that the respective fixed loci are isomorphic, as by \cite[Corollary 7.5]{BCS_Kyoto} the action of such automorphism is uniquely determined by the fixed locus. By the description of $\tau_2$ made in \cite[Section 7.4]{libassi2023git} we can see that it maps $L_i$ to $L_{i+1}$ mod 3. Moreover, $\sigma$ maps $q_i$ to $q_{i+1}$ mod 3. Therefore, $\Fix(\tau_2)\subset\hilb^2(\widetilde{\Sigma})\setminus(\bigcup\Pi_i\cup\bigcup\Lambda_{ij})$ is mapped isomorphically through $\pi^{[2]}$ to $\Fix(\sigma^{[2]})\subset\hilb^2(\Sigma)\setminus (\bigcup hc^{-1}(2q_i)\cup\bigcup hc^{-1}(q_i+q_j))$. So we are left to prove that $\Fix(\sigma^{[2]})\simeq\Fix(\widetilde{\sigma^{[2]}})$. The automorphism $\widetilde{\sigma^{[2]}}$ is defined as the only automorphism such that $\beta\circ\widetilde{\sigma^{[2]}}=\sigma^{[2]}\circ\beta$. Then as $\sigma$ maps $q_i$ to $q_{i+1}$ mod 3 it is immediate to see that $\sigma^{[2]}$ maps the irreducible component in the singular locus $\widetilde{\Sigma_i}$ to $\widetilde{\Sigma_{i+1}}$ mod 3. So $\Fix(\sigma^{[2]})\simeq\Fix(\widetilde{\sigma^{[2]}})$.
	\end{proof}
	Then we can state the following theorem.
	\begin{thm}\label{thm: A2}
		Let $C_2$ be a complex projective cubic threefold having one isolated singularity of type $A_2$ and let $Y_2$ be its associated cyclic cubic fourfold. Assume that there exist no plane $\Pi\subset Y$ such that $\Pi\cap\Sing(Y_i)\neq\emptyset$. Then the Fano variety of lines $F(Y_2)$ of $Y_2$ admits a unique symplectic resolution by an IHS manifold of $K3^{[2]}$-type $\widetilde{F(Y_2)}$. \\
		Moreover, there exists an integral lattice $T_i$, defined below, such that:\begin{enumerate}[i)]
			\item $\Pic\left(\widetilde{F(Y_i)}\right)\simeq \langle 6\rangle \oplus D_4(-1)$;
			\item there exists a non-symplectic automorphism $\tau_2\in\Aut\left(\widetilde{F(Y_i)}\right)$ whose invariant sublattice is $H^2(\widetilde{F(Y_i)}, \mathbb{Z})^{\tau_2^*}\simeq \langle 6\rangle \oplus A_2(-1)$
		\end{enumerate}
	\end{thm}
	\begin{proof}
		This is a consequence of \Cref{prop: unicity A2}, \Cref{prop: automorphismA2} and the description of $\hilb^2(\widetilde{\Sigma})$, made in \cite[Section 6]{libassi2023git}.
	\end{proof}
	\section{Final considerations}
	In this section we draw some considerations of the results in the context of nodal degenerations of cubic threefolds studied in \cite{BCS} and \cite{libassi2023git}. This interpretation is linked to the issue brought up by \cite[Section 4.2]{boissiere2023fano} in the generic nodal case.\\
	
	Consider a one parameter family $\{C_t\}_{t\neq 0}$ of smooth cubic threefolds degenerating to a nodal cubic threefold $C_0$. Then, consider the family of cyclic cubic fourfolds $\{Y_t\}$ where each $Y_t$ is branched along $C_t$ and the family of their associated Fano varieties of lines $\{F(Y_t)\}$. On each element of the family $F(Y_t)$ there exists a non-symplectic automorphism $\sigma_t$ of order 3 naturally induced by the covering automorphism. Moreover, in \cite[Section 3]{BCS} the authors showed that for $t\neq 0$ the pair $(F(Y_t), \sigma_t)$ endowed with a properly defined marking naturally lives in the moduli space $\mathcal{M}^{\rho, \zeta}_{\langle 6\rangle}$ of $(\rho, \langle 6\rangle)$-polarized IHS manifolds of $K3^{[2]}$-type. The period map $\mathcal{P}^{\rho, \zeta}_{\langle 6\rangle}$ is then surjective on the complement of the nodal hyperplane arrangement $\mathcal{H}$.\\
	
	Suppose now that $C_0$ has one isolated singularity of type $A_i$ for $i=1, \dots, 4$. Then 
	\begin{equation*}
		\lim_{t\to 0} \mathcal{P}^{\rho, \zeta}_{\langle 6\rangle}((F(Y_t), \sigma_t)=\omega_0\in\mathcal{H}.
	\end{equation*}
	In \cite{libassi2023git}, the author proved that the choice of the manifold $\widehat{\Sigma}^{[2]}=\hilb^2(\widehat{\Sigma})$ over the period $\omega_0$ with the automorphism $\hat{\tau_i}^{[2]}$ extends holomorphically the period map $\mathcal{P}^{\rho, \zeta}_{\langle 6\rangle}$ over the subloci $\Delta_3^{A_i}$. This choice was motivated by the analogy with the work of \cite{BCS} but, as proven in \Cref{thm: main SamuelPaola}, it is not the only possible choice. Indeed, the pairs $(\widehat{\Sigma}^{[2]}, \hat{\tau}^{[2]})$ and $(\widehat{F(Y_0)}, \widehat{\sigma_0})$ are equivariantly birational and, thus, if not isomorphic, they are non-separated points in $\mathcal{M}^{\rho_i, \zeta}_{T_i}$. In this case they just correspond to two different choices of Kähler chambers.
	\appendix
	
	\section{Computation of planes through the singular point}\label{sec:planes}
	Here we write the explicit computation of the equations needed to define a generic plane through the singular point $p=(1:0:...:0)$ contained in the cubic fourfold $Y$ of equation\begin{equation*}
		F(x_0, \dots, x_5)= x_0Q(x_1, \dots, x_5)+K(x_1, \dots, x_5)= 0.
	\end{equation*}
	Let $(a_0:a)$ and $(b_0:b)$ with $a,b\in\mathbb{C}^4$ be two points of $Y$. Then the plane $\Pi$ passing through these points have equation
	\begin{equation}
		(\lambda+\mu a_0+\nu b_0:\mu a+\nu b ) \quad (\lambda:\mu:\nu)\in\mathbb{P}^2.
	\end{equation}
	Imposing the condition of being in $Y$ we get:
	\begin{align*}
		&(\lambda+\mu a_0+\nu b_0)Q(\mu a+\nu b)+K(\mu a+\nu b)=\\
		&=(\lambda+\mu a_0+\nu b_0)(\mu^2Q(a)+\mu\nu B(a,b)+\nu^2 Q(b))+\\
		&+\mu^3K(a)+\mu^2\nu K^{2,1}(a,b)+\mu\nu^2 K^{1,2}(a,b)+\nu^3 K(b)=0
	\end{align*}
	using the same notations of \Cref{sec:blowuppa}. In order this to be identically zero all the coefficients of the different homogeneous components have to be trivial. Therefore:
	\begin{align*}
		&Q(a)=0\\
		&B(a,b)=0\\
		&Q(b)=0\\
		&K(a)=0\\
		&K^{2,1}(a,b)=0\\
		&K^{1,2}(a,b)=0\\
		&K(b)=0.
	\end{align*}
	Now fix $b$, the equations imply that $a=(a_1,...,a_5)$ resolves five equations. Recall now that we want the plane $\Pi$ to be non-degenerate so the points $p,$ $ a$ and $b$ are distinct. In particular, $a,b\in\mathbb{C}^4\setminus\{0\}$ so we can suppose e.g. that $a_5=1$. This shows that if $Q$ and $K$ are sufficiently generic the system has not solutions.
	\section{Computations in \Cref{prop:descrizione locale}}\label{sec:computation of descr locale}
	There are many ways to find explicit equations for a generic cubic threefold with one singularity of type $A_i$ for $i=2,3,4$. A very interesting approach is given by Heckel in his Ph.D. thesis \cite[Section 1]{Heckel}. In \textit{loc. cit.}, the author writes an explicit algorithm using the Recognition Principle \cite[Lemma 1]{BruceWall} and the Generalized Morse Lemma \cite[I, Theorem 2.47]{GLS}. This approach has the disadvantage of being too computational heavy in rapport to the results we need in \Cref{prop:descrizione locale}, so we propose here another one.
	
	Keeping the notation of \Cref{prop:descrizione locale} we want to prove the following proposition. 
	\begin{prop}
		Suppose that a cubic threefold $C_i$ defined by the equation:
		\begin{align*}
			F&=x_0q_1(x_2, x_3, x_4)+K(x_1, \dots, x_4)=\\
			&=x_0(x_2h_1(x_3, x_4)+q_1(x_3, x_4))+x_2^2h_3(x_1, x_3, x_4)+ x_2q_3(x_1, x_3, x_4)+\\
			&+k_2(x_1, x_3, x_4)=0.
		\end{align*}
		has one isolated singularity of type $A_i$ for $i=2,3,4$. Then the surface $\Gamma$ locally defined by 
		\begin{equation*}
			\overline{\phi^{3, 0}}=q_1(\widetilde{p_{13}}, p_{14})-k_2(p_{11}, \widetilde{p_{13}}, p_{14})+ p_{15}^3
		\end{equation*} has a singularity in the origin of type \boldmath $T_i$\unboldmath as defined in the following table:
		\begin{center}
			\begin{tabular}{|c|c|} 
				\hline
				$i$ & \boldmath $T_i$\unboldmath \\
				\hline
				$2$ &  $D_4$\\
				\hline
				$3$ &  $E_6$\\ 
				\hline
				$4$ & $E_8$\\ 
				\hline
			\end{tabular}
		\end{center}
	\end{prop}
	\begin{proof}
		We consider the different cases:\begin{itemize}
			\item i=2. In this case the polynomial $F(1,x_1,\dots,x_4)$ is SQH of degree $(\frac{1}{3}, \frac{1}{2}, \frac{1}{2}, \frac{1}{2})$ if and only if the coefficient of $x_1^3$ is non-trivial. As noted in the proof of \Cref{prop:descrizione locale} this implies that $\Gamma$ has a $D_4$ singularity in the origin.
			\item i=3. Clearly, the coefficient of $x_1^3$ is trivial otherwise we would be in the case $i=2$. Indeed, in this case the polynomial $F(1,x_1,\dots,x_4)$ is SQH of degree $(\frac{1}{4}, \frac{1}{2}, \frac{1}{2}, \frac{1}{2})$. So, $x_1^2$ must be multiplied by a non-trivial linear form in $x_2, x_3, x_4$. Remember that $\widetilde{p_{13}}$ is a local expression for $p_{13}$ obtained substituting it with a local expression of $q_3(p_{1,1}, p_{13}, p_{14})=p_{13}$. Therefore, in $\overline{\phi^{3, 0}}$ appears, non-trivially, either the term $p_{1,1}^4$ (if in $q_3$ depends from $x_1$) or a term in $p_{1,1}^2p_{1,4}$. In both cases $\overline{\phi^{3, 0}}$ is SQH of degree $(\frac{1}{4}, \frac{1}{2}, \frac{1}{3})$, thus $\Gamma$ has an $E_6$ singularity.
			\item i=4. Clearly we need to exclude the above cases. In this case the polynomial $F(1,x_1,\dots,x_4)$ is SQH of degree $(\frac{1}{5}, \frac{1}{2}, \frac{1}{2}, \frac{1}{2})$ and there exist no quadratic term in $x_1$. In particular, $q_3(p_{1,1}, p_{13}, p_{14})=q_3(p_{13}, p_{14})$. So, $\overline{\phi^{3, 0}}=q_1(\widetilde{p_{13}}(p_{14}), p_{14})-k_2(p_{1,1}, \widetilde{p_{13}}(p_{14}), p_{14})+ p_{15}^3$. Therefore, if $F(1,x_1,\dots,x_4)$ is SQH with weight $(\frac{1}{5})$ with respect to $x_1$ then the same is true for $\overline{\phi^{3, 0}}$, implying that it is SQH of degree $(\frac{1}{5}, \frac{1}{2}, \frac{1}{3})$, thus $\Gamma$ has an $E_8$ singularity.
		\end{itemize}
	\end{proof}
	\section{The action of the automorphism on $\Bl_{\Sigma_i}(F(Y_i))$}\label{sec:auto}
	In this section we explain better the induction of the automorphism on $\Bl_{\Sigma_i}(F(Y_i))$.\\
	
	First, by the universal property of blow-ups there exists only one automorphism $\tau$ on $\Bl_{\Sigma_i}(F(Y_i))$ commuting with the blow-up morphism $\alpha$. Consider the following diagram:
	\begin{equation*}
		\begin{tikzcd}
			\hilb^{2}(\Sigma_i) \ar[r, "\mu"] \ar [dr, "\varphi"]& \Bl_{\Sigma_i}(F(Y_i)) \ar[d, "\alpha"]\\
			& F(Y_i).
		\end{tikzcd}
	\end{equation*}
	Remember that $\mu$ is an isomorphism. Moreover, note that $\mu\circ\sigma^{[2]}\circ\mu^{-1}\in\Aut(\Bl_{\Sigma_i}(F(Y_i)))$, so if we show that this automorphism commutes with $\alpha$ we obtain that it is $\tau$. In order to show this we prove the following lemma.
	\begin{lemma}
		The automorphisms $\sigma^{[2]}$ and $\sigma$ are $\varphi$-equivariant, i.e. $\varphi\circ\sigma^{[2]}=\sigma\circ\varphi$.
	\end{lemma}
	\begin{proof}
		This is a straight-forward computation. We assume that the singular point $p\in Y$ has coordinates $(1:0:...:0)$. Moreover, we assume that $p_1:=(0:P_{11}:...:P_{15})$ and $p_2:=(0:P_{21}:...:P_{25})$ are two distinct points of $\Sigma_i$. Therefore, $\varphi(\sigma^{[2]}(p_1+p_2))=\varphi(\sigma(p_1)+\sigma(p_2))=l_{\sigma(p_1)\sigma(p_2)}$ with $l_{\sigma(p_1)\sigma(p_2)}$ the residual line of the intersection of the plane $\Pi:=\langle p, \sigma(p_1), \sigma(p_2)\rangle$ with $Y_i$. As the action of $\sigma$ is the identity on the first five coordinates we obtain that $\Pi=\sigma(\Pi')$ with $\Pi'=\langle p, p_1, p_2\rangle$, thus $\varphi\circ\sigma^{[2]}(p_1+p_2)=\sigma\circ\varphi(p_1+p_2)$. Now it remains to check the equivariance on the pairs $p_1+p_2$ where $p_1\in\Sigma_i$ and $p_2\in\mathbb{P}(T_{p_1}(\Sigma_i))$. Then, $\sigma^{[2]}(p_1+p_2)=\sigma(p_1)+d_\sigma(p_2)$ with $d_\sigma$ the differential of $\sigma$ at the point $p_1$. As the map $\sigma$ is linear its differential has the same action of $\sigma$ component-wise. Thus the same computation of the previous part show that the automorphisms $\sigma^{[2]}$ and $\sigma$ are $\varphi$-equivariant.
	\end{proof}
	Now it is easy to show that $\mu\circ\sigma^{[2]}\circ\mu^{-1}$ commutes with $\alpha$ as:
	\begin{equation*}
		\alpha\circ\mu\circ\sigma^{[2]}\circ\mu^{-1}=\varphi\circ\sigma^{[2]}\circ\mu^{-1}=\sigma\circ\varphi\circ\mu^{-1}=\sigma\circ\alpha.
	\end{equation*}
	\section{Translation of the equation}\label{sec: translation}
	In this section we write the explicit translation mentioned in the proof of \Cref{prop:blowup}.
	
	Keeping the notation of \Cref{sec:blowuppa}, note that at least one coordinate $\Bar{x}_i\neq 0$, for simplicity suppose $\bar{x}_1=1$. Then the translation to bring $\Bar{x}$ to $(0:1:0:0:0:0)$ is $t_i=x_i-\Bar{x}_ix_1$ when $i\neq 0, 1$ and $t_i=x_i$ in the other cases. So we write the \Cref{eq:cubic} in the following way:
	\begin{equation}\label{eq: uguale a BHS}
		F(t_0, \dots, t_5)= t_0Q(t_2+\Bar{x}_2, t_3+\Bar{x}_3, t_4+\Bar{x}_4)+K(t_1, t_2+\Bar{x}_2, \dots, t_5+\Bar{x}_5).
	\end{equation}
	Now, remember that $(0:1:\Bar{x}_2: \dots: \Bar{x}_5)$ satisfies $Q(\Bar{x}_2, \Bar{x}_3, \Bar{x}_4)=K(1,\Bar{x}_2, \dots, \Bar{x}_5)=0$ as $\Bar{x}\in\Sigma$. So, we can write 
	\begin{equation*}      
		\begin{split}
			Q(t_2+\Bar{x}_2, t_3+\Bar{x}_3, t_4+\Bar{x}_4)&=q_1(t_2+\Bar{x}_2, t_3+\Bar{x}_3, t_4+\Bar{x}_4)=q_1(t_2, t_3, t_4)+2t_1B_1(t_i,\Bar{x}_j) \\ &=q_1(t_2, t_3, t_4)+t_1h_1(t_2, t_3, t_4)
		\end{split}
	\end{equation*}
	and 
	\begin{equation*}      
		\begin{split}
			K(t_1, t_2+\Bar{x}_2, \dots, t_5+\Bar{x}_5)&=t_1^2h_2(t_2+\Bar{x}_2, \dots, t_5+\Bar{x}_5)+ t_1q_2(t_2+\Bar{x}_2, \dots, t_5+\Bar{x}_5)+\\&+k_2(t_2+\Bar{x}_2, \dots, t_5+\Bar{x}_5)=\\
			&=t_1^2(h_2(t_2, \dots, t_5)+k^{1,2}(t_i,\Bar{x}_j))+t_1(q_2(t_2, \dots, t_5)+\\&+2B_2t_1(t_i,\Bar{x}_j)+k^{2,1}(t_i,\Bar{x}_j))+\\&+k_2(t_2, \dots, t_5)=\\
			&=t_1^2\widetilde{h}_2(t_2, \dots, t_5)+t_1\widetilde{q}_2(t_2, \dots, t_5)+k_2(t_2, \dots, t_5).
		\end{split}
	\end{equation*}
	If we substitute the expressions of $Q$ and $K$ in \Cref{eq: uguale a BHS} we see that it has the same form of \cite[Equation (3.2)]{boissiere2023fano}.
	\bibliographystyle{alpha}
	\bibliography{bibliographystableFano} 
\end{document}